\documentclass[12pt]{article}
\usepackage{latexsym,amssymb,amsmath,amsfonts,amsthm,graphicx,url}
\numberwithin{equation}{section}
\newtheorem{theorem}{Theorem}[section]
\newtheorem{prop}{Proposition}[section]
\newtheorem{lemma}[theorem]{Lemma}

\theoremstyle{definition}

\def\beq{ \begin{equation} }
\def\eeq{ \end{equation} }
\def\mn{\medskip\noindent}
\def\ms{\medskip}

\def\bn{\bigskip\noindent}
\def\fns{\footnotesize}

\def\square{\vcenter{\vbox{\hrule height .4pt
  \hbox{\vrule width .4pt height 5pt \kern 5pt
        \vrule width .4pt} \hrule height .4pt}}}

\def\ep{\epsilon}

\def\ZZ{\mathbb{Z}}

\def\To{\Rightarrow}

\def\clearp{}

\begin{document}

\title{Continuous phase transitions in the $k$-creation\\
 process without stirring in $d=1$}
\author{Rick Durrett \\
\small James B. Duke Emeritus Professor of Math,
 Durham, NC 27705}

\date{\today}						

\maketitle

\begin{abstract}
This paper is a companion to one in which we prove there is a discontinuous phase transitions in the $k$-creation process with fast stirring in $d=1$ when $k \ge 2$. Dickman and Tom\'e (1991) introduced models on $\ZZ$ in which $k$ consecutive occupied sites give birth at rate $\lambda$ and  individual particles die at rate 1. Here, we show that without stirring the models with $k\ge 2$ have  qualitative properties much like the contact process, which is the case $k=1$. The critical value can be charracterized by the speed of interface when the process starts from the initial configuration $(-\infty,0]$. The process dies out at the critical value.  In the supercritical phase the complete convergence theorem holds which implies there is only one nontrivial stationary distribution, and convergence to the limit occurs exponentially rapidly. In the subcritical phase, the process dies out exponentially fast starting from any finite set.
\end{abstract}

\section{Introduction}

\mn
In the $k$-creation model in continuous time, if the configuration is $\xi \in \{0,1\}^\ZZ$ then  the death rate is $d(x,\xi)\equiv 1$ when $\xi(x)=1$ and the birth rate $b_k(x,\xi)$ is
$$
 (\lambda/2) \left(1_{(\xi(x+i)=1 \hbox{ \fns for all $1\le i  \le k$})} 
+ 1_{(\xi(x-i)=1 \hbox{ \fns for all $1\le i\le k$})} \right)
$$
when $\xi(x)=0$. When $k=1$ this is the basic contact process. Our focus here will be on the behavior of the process with $k\ge 2$. Generalizing Liggett's (1999)  proof of the Beuidenhout Grimmett (1990) theorem we prove that for any $k$,

\begin{theorem} \label{domOP}
Suppose there is positive probability the system survives starting from $\{1, \ldots k\}$. Then for any $p< 1$ the $k$-creation process, when viewed on suitable space and time scales, it dominates oriented percolation in which sites are independently open with probability $p$.
\end{theorem}

\mn
A precise statement can be found in Proposition \ref{sandt} below, with other more useful formulations in Propositions \ref{compOP0} and \ref{compOP2}.  The key to proving our resulst is to look at $\zeta_t =$ the sites in $\xi_t$ that can give birth, i.e. the $x\in \xi_t$ so that $\{x, x+1, \ldots x+k-1\} \subset \xi_t$ or $\{x-k+1, \ldots x-1, x\} \subset \xi_t$,
Once Theorem \ref{domOP} is proved, we immediately get three corollaries

\begin{theorem} \label{crdies}
The $k$-creation  process dies out at its critical value. 
\end{theorem}

\mn
Bezuidenhout and Gray (1994) have proved this result for general attraactive processes, but they assume that there is positive probability of survival from a singleton, which does not hold here when $k \ge 2$.

\begin{theorem} \label{supercr}
Suppose $\lambda> \lambda_c$ and let $\tau^A = \inf\{ t \ge 0: \xi^A_t =  \emptyset \}$ be the extinction time starting from $A$. Then there are constants $C_1$, $C_2$, $\gamma_1$ and $\gamma_2$ so that
\begin{align}
&P( t \le \tau^A < \infty)  \le C_1e^{-\gamma_1t},
\label{scexp1}\\
&P( \tau^A < \infty) \le C_2e^{-\gamma_2|A|}.
\label{scexp2}
\end{align}
\end{theorem}

\mn
Liggett (1999) also proves the complete convergence theorem for the contact process (see Theorem \ref{cct} below for the statement) but his proof requires that the contact process is an additive process, i.e., the process can be built on a special {\bf graphical representation} and hence has a set-valued dual process $\xi^A_t$ which satisfies
$$
P( \xi^A_t \cap B \neq \emptyset ) = P( A \cap \zeta^B_t \neq \emptyset ) .
$$
The name {\bf additive process} comes from the fact that when built on this graphical representation
$$
\xi^{A \cup B}_t = \xi^A_t \cup \xi^B_t .
$$
Letting $1\le i <k$ and taking $A=\{1, \ldots i\}$ and $B=\{i+1, \ldots k\}$ shows that the $k$-creation process with $k\ge 2$ is not additive, since births cannot occur in $\xi^A_t$ or $\xi^B_t$ but they can in $\xi^{A \cup B}_t$. For more about additive processes, see Harris (1976), Griffeath (1979), or Section 1.3 in Durrett (2027).

To prove the complete convergence theorem  for the $k$-creation model we will use coupling. The first step is Theorem 2.1 in Durrett (1980), which holds for a class of finite range growth models described in the paper. Let $\bar r_t = \max \xi^{(-\infty,0]}_t$.

\begin{theorem} \label{edgesp}
If $\lambda > \lambda_c$ then there is a constant $\alpha(\lambda)>0$ so that as $t\to\infty$ the right edge $\bar r_t/t \to \alpha(\lambda)$ almost surely.
\end{theorem}

\mn
The existence of the limit follows from subadditivity (and does not require $\lambda>\lambda_c$). The positivity of the constant follows from Theorem \ref{domOP}. 
Conversely, we have

\mn
\begin{theorem} \label{zeroatcr}
$\alpha(\lambda_c)=0$, and $\alpha(\lambda) \to 0$ as $\lambda \downarrow \lambda_c$. 
\end{theorem}

\mn
Before we turn to the proof of the complete convergence theorem, we have one last corollary of ``edge speeds characterize the critical value.''

\begin{theorem} \label{expsub}
If $\alpha(\lambda)<0$ and $I=[1,k]$ the there is a constant $\gamma>0$ so that $P( \xi^I_t \neq \emptyset ) \le C \exp^{-\gamma t}$, and hence $\alpha(\lambda)=-\infty$.
\end{theorem}

\mn
Although we have not given all the necessary definitions: if $E\bar r^M_M<0$ and $E\bar \ell^M_M>0$ the process dies out exponentially fast since $\bar r_t < \bar \ell_t$ implies that the proces has died out.

Careful readers will note that we have not ruled out the possibility that $\alpha(\lambda)=0$ on an interval of positive length. In the contact process or more generally additive processes Lemma 4.1 in Durrett (1980) implies that if $B \subset (-1,\infty]$ then
\beq
E(r^{B \cup \{0\}}_t - r^B_t) \ge E(r^{(-\infty,0]} - r^{(-\infty,1]} \ge 1
\label{Lig}
\eeq
which leads to  $\alpha(\lambda+\delta)-\alpha(\lambda) \ge \delta t$. Not only doies the proof of \eqref{Lig} breakdown when $k \ge 2$ but the result is false. if we let $B=(-\infty,-10]$ then it is unlikely that process catches up with the one at 0 before it dies when $\alpha(\lambda)$ is small.

\medskip
{\bf Complete convergence theorem.} To motivate the coupling proof for the $k$-creation model, we recall the proof for  contact process given in Durrett (1990).  Let $\xi^0_t$ be the process with $\xi^0_0 = \{0\}$, let $\xi^1_t $ be the process with $\xi^1_0 = \ZZ$  and let $\bar\ell_t = \min \xi_t^{[0,\infty)}$. Then when the processes are defined on the same space using the graphical representation 
\beq
\xi^0_t = \xi^1_t \cap [\bar\ell_t, \bar r_t] \quad\hbox{on $\xi^0_t \neq\emptyset$}
\label{ccouple}
\eeq
To give this a chance of holding when $k\ge 2$ we need to replace 0 by an interval $I=[a,b]$ with $b-a \ge k-1$ and modify the definitions of the edge processes to be  $\bar r_t = \max \xi^{(-\infty,b]}_t$ and $\bar\ell_t = \min \xi_t^{[a,\infty)}$ but even if we do this then the coupling fails to satisfy \eqref{ccouple}. Consider for simplicity the case $k=2$ and look at the following situation

\begin{center}
\begin{tabular}{cccc}
& $m-1$ & $\bar r_t = m$ & $m+1$ \\
$\xi^1_t$ & 0 & 1 & 1 \\
$\xi^I_t$ & 0 & 1 & 0
\end{tabular}
\end{center}

\noindent
In $\xi^1_t$ a birth can occur at $m-1$ that will not happen at $\xi^I_t$.
Somewhat surprisingly the coupling can be fixed by looking at the leftmost and rightmost particles that can give birth:
\begin{align*}
\tilde r_t &= \max\{ x: \xi^{(-\infty,b]}_t(y) = 1 \hbox{ for all $y \in [x-k-1,x]$} \} \\
\tilde \ell_t &= \min\{ x: \xi^{[a,-\infty)}_t(z) = 1 \hbox{ for all $z \in [x,x+k-1]$} \}
\end{align*}
For then we have

\begin{lemma} \label{kcouple}
In the $k$-creation model if $I=[a,b]$ with $b-a \ge k-1$ then 
$$
\xi^I_t \cap [\tilde\ell_t, \tilde r_t]  =  \xi^1_t \cap [\tilde\ell_t, \tilde r_t] \quad\hbox{on $\zeta^I_t \neq\emptyset$}.
$$
\end{lemma}

\noindent
We need the intersection on the left since we may not have $\xi^I_t=\xi^1_t$ in $[\bar \ell_t, \bar r_t] - [\tilde\ell_t, \tilde r_t]$. In this region we have 1's that belong to intervals of length $k-1$ and cannot reproduce. The next result shows that when $\lambda>\lambda_c$ these regions have width $o(t)$

\begin{lemma} \label{tildesp}
If $\lambda > \lambda_c$ then $\tilde r_t/t \to \alpha(\lambda)$ and $\tilde \ell_t/t \to -\alpha(\lambda)$.
\end{lemma}

\mn
Combining the last two lemmas leads easily to 

\begin{theorem} \label{cct}
If $\lambda > \lambda_c$ then
$$
\xi^A_t \To P(\tau^A<\infty) \, \delta_\emptyset
 + P(\tau^A=\infty)\, \xi^1_\infty.
$$
\end{theorem}

The remainder of the paper is devoted to proofs. In Section \ref{sec:prelim} we construct the processes from a family of Poisson processes and we introduce some general results that will be needed in the proof: a zero-one law for the tail $\sigma$-field of the Poisson processes, and two results that allow us to conclude the processes have positive correlations. In Section \ref{sec:Prop3.2}, we prove Proposition \ref{sandt} which is the key to Liggett's (1999) version of the block construction of Bezuidenhout and Grimmett (1990). In Section 4 we use the block construction to prove Theorems \ref{crdies} and \ref{supercr}. In Section \ref{sec:Th456} we sketch the proof of Durrett's (1980) result (Theorem \ref{edgesp}) about the right edge and use the ideas to prove Theorems \ref{zeroatcr} and \ref{expsub}. In Section \ref{sec:Th456}, we introduce the right-most infectious site in the $k$-creation model, and use it to prove the complete convergence theorem via a combination of a limit theorem for $\tilde r_t/t$ (Theorem \ref{tildesp}) a coupling (Theorem \ref{kcouple}), and a consequence of the block construction. The proofs in this paper are old fashioned and restricted to one dimension, but the results are a rare example of a thorough analysis of an interacting particle system that is neither reversible reversible nor an additive process. 

\clearp

\section{Preliminaries} \label{sec:prelim}

Our first step is to construct the $k$-creation process from three families of Poisson processes

\begin{itemize}

\item
For $x \in \ZZ$, let $\{ D^x_n, n \ge 1\}$ be Poisson processes with rate 1.  At the arrival times of these Poisson processes we write a $\bullet$ at $x$ that kills any partilce on the site.

\item
For $x \in \ZZ$, let $\{B^{x,+}_n, n \ge 1 \}$ and let $\{B^{x,-}_n, n \ge 1 \}$ be Poisson processes with rate $\lambda/2$. At the arrival times of $B^{x,+}_n$ we draw an arrow from $x+k$ to $x$ to indicate that there will be a birth at $x$ if $x+1. \ldots x+k$ are all occupied and $x$ is vacant. At the arrival times of $B^{x,-}_n$ we draw an arrow from $x-k$ to $x$ to indicate that there will be a birth at $x$ if $x-1. \ldots x-k$ are all occupied and $x$ is vacant.

\end{itemize}

\noindent
These Poisson processes allow us to construct all of the processes $\xi^A_t$ starting from $\xi^A_0=A$ on the same space $\Omega$ so that if $A \subset B$ then $\xi^A_t \subset \xi^B_t$, which shows that they are {\bf attractive}.

\bn
{\bf Zero-one law}

\mn
Let ${\cal F}'_n$ be the $\sigma$-field generated by the points $(x,t)$ outside of $-n \le x \le n$, $t\le n$. ${\cal T} = \cap_{n\ge 1} {\cal F}'_n$ be the tail $\sigma$-field. Since the Poisson processes are independent, it is a standard fact that

\begin{lemma} \label{trivial}
${\cal T}$ is trivial.
\end{lemma}

\begin{proof} To make this easier to see let $Y_{m,n}$ be independent random variables that consist of the Poisson processes at $m \in \ZZ$ for times $t \in (n,n+1)$ for all nonnegative integers $n$. \end{proof}

\bn
{\bf  Positive correlations} 

\mn
In the theory of percolation a result first proved by Harris (1960) is very useful. The treatment here is based on pages 8--9 in Liggett (1999) and Section 2.2 in Grimmett (1991). Suppose we have independent random varaibles $X= \{X_1, \ldots X_n\}$ with $X_i\in \{0,1\}$.

\begin{theorem} \label{Harris60}
If $f$ and $g$ are increasing then
$$
E(f(X)g(X)) \ge Ef(X)Eg(X)
$$
\end{theorem}

\mn
If $f$ and $g$ are decreasing applying this result to $-f$ and $-g$ gives the same conclusion. Considering $f=1_A$ and $g=1_B$ we see that if both $A$ and $B$ are increasing or both are decreasing then
$$
P(A \cap B) \ge P(A)P(B)
$$
By discretizing time tand taking limits, this result can be extended to increasing or decreasing events of our Poisson construction. Another way to derive positive correlations for the $k$-creation model  is

\bn
{\bf Harris' (1977) theorem} 

\mn
Let $E$ be a finite set with a partial ordering $\le$. Call $f$ {\bf increasing} if $x<y$ implies $f(x) \le f(y)$ and let $F_{inc}$ be the set of all increasing functions. We say that $\mu$ has {\bf positive correlations} if $\mu(fg) \ge \mu(f)\mu(g)$ for all $f,g \in F_{inc}$. Let ${\cal M}_{pc}$ be the set of $\mu$ with positive correlations.
Let $X_t$, $t\ge 0$ be a Markov process with (right-continuous) step-function paths taking values in the finite state space $E$, a transition probability $p(t,x,y)$, and semigroup $T_tf(x) = \sum_y p(t,x,y) f(y)$. Call $X_t$ or $T_t$ {\bf monotone} if $T_tF_{inc} \subset F_{inc}$. Let $U_t\mu(y) =\sum_x \mu(x) p(t,x,y)$ be semigroup acting on measures.

\begin{theorem} \label{Harrispc}
Let $X_t$ be an attractive process in a finite partially ordered state space $E$.
In order that $U_t {\cal M}_{pc} \subset {\cal M}_{pc}$ for each $t>0$ it is necessary and sufficient that each jump is up or down in the partial order on $E$.
\end{theorem}

\mn
The condition that each jump is up or down means that the stirring rate muct be 0.

\clearp

\section{The first step: Proof of Proposition \ref{sandt}} \label{sec:Prop3.2}

\mn
In this section we introduce the key idea due to Bezuidenhout and Groimmett (1990). To prove the result. we follow Section 2 in Part I of Liggett (1999) restricting to $d=1$ which simplifies the proof. We use $\xi$ instead of $A$ to denote the process, but otherwise the plan is the same as in Liggett's book. As announced on page 45 is

\begin{quote}
The first step is to show that if the contact process survives then the contact process restricted to a large space-time box $[-L,L]^d \times [0,T]$ has the property that there are many infected sites on various parts of the boundary with high probability if the initial configuration is large enough.
\end{quote}

To prepare for the proof for the $k$-creation process with $k \ge 2$, we will begin by describing Liggett's proof for the contact process in the case $d=1$.

\begin{lemma} \label{P2.1}
Suppose $\xi_t$ survives. Then
$$
\lim_{n\to\infty} P(|\xi_t^{[-n,n]}\neq\emptyset \hbox{ for all $t\ge 0$}) = 1
$$
\end{lemma}

\mn
Let ${}_L\xi_t$ be the process restricted to $[-L,L]$. Let 
$$
\Omega_{\infty}
= \{\xi_t\neq\emptyset \hbox{ for all $t\ge 0$}\} \qquad
\Omega_0  = \{\xi_t =\emptyset \hbox{ for some $t$}\}
$$ 

\begin{lemma} \label{P2.2}
For every finite $A$ and every integer $N \ge 1$
$$
\lim_{t\to\infty}\lim_{L\to\infty} P( |{}_L\xi_t^A| \ge N)
= \lim_{t\to\infty} P( |\xi_t^A| \ge N)
=P^A(\Omega_{\infty})
$$
\end{lemma}

\mn
Define $S(L,T) = \{ (x,t) \in \ZZ \times [0,T] : |x|=L \}$ and
$$
{}_L\xi^A = \cup_{t\ge 0} ( {}_L\xi^A_t \times \{t\}) \subset \ZZ \times [0,\infty)
$$
Let $N^A(L,T)$ be the maximum number of points in $S(L,T) \cap {}_L\xi^A$ with the following property: if $(x,s_1)$ and $(x,s_2)$ are $\in S(L,T)$ then $|s_1-s_2|\ge 1$.

\begin{lemma} \label{P2.8}
Suppose $L_j \uparrow \infty$ and $T_j \uparrow\infty$. For any $M,N$ and for any finite $A$
\begin{align*}
&\limsup_{j\to\infty} P(N^A(L_j,T_j) \le M) \cdot  P(|{}_{L_j}\xi^A_{T_j}| \le N) \\
&\le \limsup_{j\to\infty} P(N^A(L_j,T_j) \le M, |{}_{L_j}\xi^A_{T_j}| \le N)
\le P^A(\Omega_0)
\end{align*}
\end{lemma}

\mn
The final part of the argument is to use positive correlations to show that we can have the good events occur in one half of the space time set that we are free to choose.

\begin{lemma} \label{P2.6}
For every $n,N \ge 1$ and $L > n$ 
$$
 P( |{}_L\xi_t^{[-n,n]} \cap [0,L] |  \le N)
\le P( |{}_L\xi_t^{[-n,n]}|  \le 2N)^{1/2}
$$
\end{lemma}

\mn
Define the right side of $S(L,T)$ by  $S_+(L,T) = \{ (x,t) \in \ZZ \times [0,T] : x=L \}$ and
let $N^A_+(L,T)$ be the maximum number of points in $S_+(L,T) \cap {}_L\xi^A$ with the following property: if $(x,s_1)$ and $(x,s_2)$ are $\in S_+(L,T)$ then $|s_1-s_2|\ge 1$.

\begin{lemma} \label{P2.11}
For any $L$, $M$, $T$ and $n<L$
$$
[P( N^{[-n,n]}_+(L,T) \le M)]^2 \le P( N^{[-n,n]}(L,T) \le 2M)
$$
\end{lemma}

\mn
Define two families of events
\begin{align*}
&\hbox{top}(x)  = {}_{L+2n}\xi^{[-n,n]}_{T+1} \supset x+ [-n,n] \\
&\hbox{side}(y,t)  = {}_{L+2n}\xi^{[-n,n]}_{t+1} \supset y + [-n,n]
\end{align*}
The next result is Theorem 2.12 in Part I of Liggett (1999)

\begin{prop} \label{sandt}
If $\xi_t$ survives then for every $\ep >0$ there are choices of $n$, $L$, $t$ so that
\begin{align}
& P( \hbox{top}(x)\hbox{ for some $x \in [0,L]$} ) \ge 1-\ep 
\label{topE}\\
& P(\hbox{side}(L+n,t) \hbox{ for some $0 < t < T$}) \ge 1-\ep
\label{sideE}
\end{align}
\end{prop}

Once Proposition \ref{sandt} is established  then the comparison with oriented percolation and the two corollaries follows exactly as in the argument for the contact process given in Liggett (1999). We describe the ideas involved in Section \ref{sec:Th123}.

\bn
{\bf Proofs of Lemmas \ref{P2.1} - \ref{P2.11} for the $k$-creation model.}

\mn
Here $\xi_t$ is the set of infected (or occupied) sites while $\zeta_t \subset \xi_t$ are the {\bf infectious sites} which are the $x \in \xi_t$ so that $\{x, x+1, \ldots, x+k-1\} \subset \xi_t$ or  $\{x, x-1, \ldots, x-(k-1) \} \subset \xi_t$. These are the sites that are capable of infecting other sites. Note that these sites can give birth onto vacant sites in at most one direction so the total birth rate is $\le (\lambda/2)| \zeta_t|$.

\begin{proof}[Proof of Lemma \ref{P2.1}]
Let $E_m$ be the event that the $k$-creation process starting from $\{(k-1)m+1, \ldots km\}$ survives. $\{E_m \hbox{ infinitely often}\} \in {\cal T}$, the tail $\sigma$-field defined in Section \ref{sec:prelim}, and its probability is lower bounded by $\limsup_{m\to\infty} P(E_m) >0$, so by Lemma \ref{trivial} 
$P(E_m\ { i.o.})=1$. Let $N_\ell$ be the number of events $E_1, \ldots E_\ell$ that occur. $N_\ell \to\infty$ with probability 1, so $P( N_\ell \ge 1) \to 1$. Since $N_\ell \ge 1$ implies
the process survives starting from $[1,\ell k]$ survives the desired result follows. 
\end{proof}

\mn
The infectious sites $\zeta_t$ will replace $\xi_t$ in the next four of Liggett's  lemmas (and of course also  in the associated definitions). Otherwise the proofs are almost identical to the ones in Liggett (1999).

\begin{proof}[Proof of Lemma \ref{P2.2}]
Since $\zeta_t^A =\cup_L \, {}_L\zeta_t^A$ it follows that
$$
\lim_{L\to\infty} P( |{}_L\zeta_t^A| \ge N ) =  P( |\zeta_t^A| \ge N )
$$
In the $k$-creation model the total birth rate is $\le (\lambda/2) |\zeta^A_t|$ while the total death rate of infectious sites is $|\zeta^A_t|$. The probability a death is the next event to occur an infectious site is $\ge 1/(1+\lambda/2)$. Iterating and noting that once all the infectious sites are dead the process is doomed to die out.
\beq
P( \Omega_0|{\cal F}_s) 
\ge \left( \frac{1}{1+\lambda/} \right)^{|\zeta^A_s|}
\label{deadpr}
\eeq
The right hand side is very small, but it is positive, so L\'evy's 0-1  implies 
$P( \Omega_0 |{\cal F}_s) \to 1_{\Omega_0}$. This means that $P( \Omega_0 |{\cal F}_s) \to 0$ on $\Omega_\infty$ and it follows that $|\zeta^A_t| \to\infty$ on $\Omega_\infty$.
\end{proof}

\begin{proof}[Proof of Lemma \ref{P2.8}]
The first inequality follows from the fact that decreasing events are positively correlated.
The proof of the second is similar to the proof of Lemma \ref{P2.2} but has a new wrinkle that comes from the fact that on the side we have selected points that are separated by more than 1. To deal with that issue, suppose that the points on the side $x=L$ occur at times $s_1, s_2, \ldots s_\ell$ and let $I=\cup_{j=1}^\ell \{L\} \times (s_i-1,s_i+1)$ so all the points in $\{x=L\} \cap {}_L\zeta^A$ are contained in $I$. The set $I$ has Lebesgue measure $2\ell$ so 
$$
P(\hbox{ there is no birth to the right from $I$})  \le e^{-\lambda \ell}.
$$

Turning to the gaps between or above all the intervals in $I$, if the ``gap'' has length $u$ then the probability that either (i) there is no birth to the right or (ii) there is one but a recovery symbol comes before it occurs is
$$
e^{-u\lambda/2} + \int_0^u (\lambda/2) e^{-s\lambda/2} [1-e^{-s}] \, ds 
= 1 - \int_0^u (\lambda/2) e^{-(1+\lambda/2) s)} \, ds \ge \frac{1}{1+\lambda/2} ,
$$
so the probability of a birth to the right when $N^{[-n,n]}_x(L,T) = \ell$ is
$$
\le \left( \frac{ e^{-\lambda} }{1+\lambda/2 } \right)^\ell
$$

Let ${\cal F}_{L,T}$ be the $\sigma$-algebra generated by the Poisson processes in $[-L,L] \times [0,T]$. Applying the argument for Lemma \ref{P2.2} to the top of the region, we see that 
$$
P( \Omega_0|{\cal F}_s ) \ge \left(\frac{ e^{-\lambda}}{1+\lambda/2} \right)^k
$$
on $\{ N^A(L,T) + |{}_L \zeta^A_t| = k \}$, which proves the desired conclusion. 
\end{proof}

\begin{proof}[Proof of Lemma \ref{P2.6}]
Let $X_1= |{}_L\zeta_t^{[-n,n]} \cap [0,L)$, $X_2= |{}_L\zeta_t^{[-n,n]} \cap (-L,0]$.
$$
 P( |{}_L\zeta_t^{[-n,n]}|  \le 2N) \ge P( X_1 \le N, X_2 \le N) \ge P(X_1 \ge N)^2
$$
by positive correlations.
\end{proof}

\begin{proof}[Proof of Lemma \ref{P2.11}]
Use positive correlations as in the previous proof.
\end{proof}

\bn
{\bf For the $k$-creation model, Lemmas \ref{P2.1} - \ref{P2.11} imply Proposition \ref{sandt}}

\begin{proof} 
Once we substitute $\zeta_t$ for $\xi_t$ the proof is almost the same as that for the contact process given on pages 50-51 of Liggett (1999). To simplify, we will omit the details  of how to pick $\delta$ to end up with an error probability of exactly $\ep$. 
Given $\delta>0$ to be chosen later, we use Lemma \ref{P2.1} to choose $n$ so that
$$
P^{[-n,n]}(\Omega_\infty) > 1 -\delta^2
$$

Choose $N$ so large so that any $N$ points in $\ZZ$ will contain a subset of size $N'$ separated by a distance of $2n+1$, where $N'$ is such that in $N'$ independent trials with success probability $P({}_n \xi_1^0 \supset [-n,n])$ we will have at least one success with probability at least $1-\delta$. When it comes to $M$ it is enough to require that in $M$ independent trials with success probability $P({}_{[0,2n]} \xi_1^0 \supset 0,2n])$ we will have at least one success with probability at least $1-\delta$.

Using Lemma \ref{P2.2} and Lemma \ref{P2.8} with $M$ and $N$ replaced by  $2M$ and $2N$ we can find $L_j \uparrow \infty$ and $T_j \uparrow \infty$ so that
\begin{align*}
& P( |{}_{L_j} \xi^{[-n,n]}_{T_j}| > 2N ) = 1 -\delta \qquad\hbox{for all $j\ge 1$}\\ 
& P( N^{[-n,n]}(L_j,T_j) > 2M) > 1 -\delta \qquad\hbox{for some $j\ge 1$}
\end{align*}
Taking $L=L_j$ so that the second result holds, using Lemmas \ref{P2.6} and \ref{P2.11} and then taking $T=T_j$ we can improve these conclusions to 
\begin{align*}
& P( |{}_{L} \xi^{[-n,n]}_{T_j} \cap [0,L)| > N )\ge  1 -\delta^{1/2} \\ 
& P( N^{[-n,n]}(L,T) > M) > 1 -\delta^{1/2}
\end{align*}
Using the definitions of $M$ and $N$ with the fact that Poisson points in disjoint parts of the plane are independent proves the two conclusions in Proposition \ref{sandt} 
\end{proof}

\clearp

\section{Proofs of Theorems \ref{crdies} and \ref{supercr}} \label{sec:Th123}

This is easy ,once we transform Proposition \ref{sandt} into a more useful result.
The first step is to use Proposition \ref{sandt} twice to get a fully occupied copy of $[-n,n]$ displaced in space and time.
$$
\hbox{cover}(x,t) = {}_{2L-3n}\xi_t^{[-n,n]} \supset x +[-n,n]
$$

\begin{prop} \label{compOP0}
Suppose the conditions in Proposition \ref{sandt} hold. Then for every $\ep$ there are choices of $L$ and $T$ so that
$$
P( \hbox{cover}(x,t)  \hbox{ for some $x\in [L+n,2L+n]$, $t \in [T,2T]$} ) \ge 1-\ep
$$
\end{prop}

\mn
This is Proposition 2.20 in Liggett (1999). Two more modification give Proposition 2.23 which gives the comparison with oriented site percolation in a more convenient form. Taking advantage of a result of Liggett, Schonmann, and Stacey 1997) he is able to compared with the traditional process with independent sites which he calls $B_n$. He works on a slightly different graph but it is just a linear transformation of the usual one. 
$$
P( x \in B_{n+1} |B_0, B_1, \ldots B_n) = p 
\quad \hbox{ if $B_n \cap \{x-1,x\} \neq \emptyset$}
$$
and 0 otherwise . For simplicity we state the result only in $d=1$

\begin{prop} \label{compOP2}
Suppose the conditions in Proposition \ref{sandt} hold. For every $p<1$ there are choices of $N$, $a$, and $b$ so that the following holds. If the initial conditions $B_0$ and $A$ have the property that $j \in B_0$ implies
$$
A \supset x+[-n,n] \hbox{  for some $x\in [(4j-1)a,(4j+1)a]$}
$$
then $\xi^A_t$ and $B_k$ with parameter $p$ can be coupled so that $j\in B_k$ implies 
$\xi^A_t \supset x+[-n,n]$ 
$$
\hbox{ for some} \ (x,t) \in [(4j-2k-1)a,(4j+2k+1)] \times [5kb,(5k+1)b]
$$
\end{prop}

\mn
{\bf Proposition \ref{compOP2} implies Theorem \ref{crdies}.} The idea is old and very simple. If Theorem \ref{compOP2} holds when $\lambda = \lambda_c$ then it holds with $1-\ep$ replaced by $1-2\ep$ for $\lambda'< \lambda_c$, which is a contradiction if $\ep$ is small enough so that there is percolation when $p=1-2\ep$. See pages 51--54 in Liggett (1999).

\mn
{\bf Theorem \ref{supercr} is proved in Liggett's Theorem 2.30}. Intuitively, given the comparisons between process of interest and oriented percolation, it suffices to prove the result for oriented percolation with $p$ close to 1, which is easy.

\clearp

\section{Proofs of Theorems \ref{edgesp}, \ref{zeroatcr}, and \ref{expsub}}
 \label{sec:Th456}

We will not give a complete proof of Theorem \ref{edgesp}. However, we will describe the main ideas becuase some of them are needed in the proof of the other two result.

\begin{proof} [Proof of Theorem \ref{edgesp}]
Using the notation of Durrett (1980), we let $R = (-\infty,0]$. To analyze the growth of $\bar r_t$, we will introduce a family of {\bf reset approximations}
$\xi^{R,M}_t$ which start from $\xi^{R,M}_0= (-\infty,0]$

\mn
(i) On the time intervals $[0,M), [M,2M), \ldots$ the process evolves according to the rules of the $k$-creation process.

\mn
(ii) At times $M, 2M, \ldots$ we reset to 1 the values at all sites to the left of the rightmost one.

\mn
Since the process is attractive we can define $\xi^R_t$ and $\xi^{R,M}_t$ on the same space so that  $\xi^R_t(x) \le \xi^{R,M}_t(x)$ for all $x$. If we let $\bar r^M_t =
\max\{ y : \xi^{R,M}_t(y) = 1\}$ then $\bar r_t \le \bar r^M_t$. The asymptotic behavior of
$\bar r^M_t$ is easy to determine because $\bar r^M_{kM} - \bar r^M_{(k-1)M}$, $k \ge 1$ are independent and identically distributed so if the mean is finite
$$
k^{-1}\bar r^M_{kM} \to E\bar r^M_M = E\bar r_M.
$$
An argument given on page 893 of Durrett (1980) bounds the movement at intermediate times, so that we can conclude 
$$
\limsup_{t\to\infty} t^{-1}\bar r^M_t \le  E\bar r_M/M
$$
Since $\bar r_t \le \bar r^M_t$ we can replace $\bar r^M_t$ by $\bar r_t$ in the last result. Taking the infimum over $M$ gives 
$$
\limsup_{t\to\infty} t^{-1}\bar r_t \le \inf_{M\ge 1} E\bar r_M/M \equiv \alpha
$$
If $\alpha=-\infty$ we are done, so we will suppose $\alpha> -\infty$. From the definition of $\alpha$ it is clear that $\alpha \le E\bar r_1 <\infty$. 

At this point some  will be reminded of the proof of the subadditive ergodic theorem in Section 6.4 of Durrett (2019). If not, the reader should become familiar with that proof since the proof in Durrett (1980) follows its outline. The next two steps are to show

\ms
(i) $\bar r_t \to \alpha$ in $L^1$

\ms 
(ii) $\liminf_{t\to\infty} t^{-1}\bar r_t \ge \alpha$

\mn
The first claim is easy enough to be a homework problem in the first graduate probability class, but the seond one requires a fair anount of work, so we leave it to the reader to find the details on pages 894--896 of Durrett (1980).
\end{proof}

\begin{proof} [Proof of Theorem \ref{zeroatcr}]
Suppose $\alpha(\lambda_c)>0$. This implies that $\bar r_t \to\infty$ as $t\to\infty$ and hence $\inf_t \bar r_t \equiv V > - \infty$ almost surely. This implies that there is a positive integer $v< \infty$ so that $P(V \ge -v)>0$. Applying this reasoning to $\bar\ell_t$ we see that $\sup_t \bar\ell_t =_d  - V$ so if we start from $[-2v,2v]$ occupied then 
$P( \bar\ell_t < \bar r_t \hbox{ for all $t\ge 0$})>0$ so the process survives with positive probability contradicting Theorem \ref{crdies}.

To prove the second conclusion, let $\alpha_M(\lambda) = E\bar r_M/M$ and
$a_m(\lambda) = \min\{ \alpha_M(\lambda) : 1 \le M \le m \}$. $a_m(\lambda)$ is continuous and decreases to $\alpha(\lambda)$, so $\alpha(\lambda)$ is upper semicontinuous. Since $\alpha(\lambda)$ nondecreasing it must be right continuous.
\end{proof}

\begin{proof}[Proof of Theorem \ref{expsub}]
If $\alpha(\lambda)<0$ then we can pick $M$ so that $E\bar r^M_M = -\mu<0$ and hence $E\bar\ell^M_M =\mu >0$. The growth of $\bar r^M_t$ is bounded above by a Poisson process with rate $\lambda/2$ so $E\exp(\theta \bar r^M_M) < \infty$ for small $\theta>0$. Standard large deviations results imply (see Section 2.7 in Durrett (2019)) that $P( \bar r^M_{kM} > - k\mu/2) \le \exp(-\kappa k)$. Since $\bar r_{kM} \le \bar r^M_{kM}$ and $\bar \ell_{kM} \ge \bar \ell^M_{kM}$ it follows that the process has died out when $\bar r^M_{kM} <\bar \ell_{kM}$ which completes the proof.
\end{proof}

\clearp

\section{Proofs of Lemmas \ref{kcouple} and \ref{tildesp} and Theorem \ref{cct}}
\label{sec:L67Th8}

\begin{proof}[Proof of Lemma \ref{kcouple}] The coupling holds at time 0, so we only have to show that discrepancies cannot be created before $\zeta^I_t$ dies out.. While $\zeta^I_t \neq \emptyset$. $\tilde\ell_t$ and $\tilde r_t$ are in $\zeta_t$ so births cannot occur into $[\tilde\ell_t + k,\tilde r_t -k]$ from outside. The next figure gives a picture of the situation siwth $k=3$

\begin{center}
\begin{tabular}{ccccccccccccccccc}
 $\xi^1_t$ & 0 & 1 & 1 & 1 & 0 & 1 & 0 & 1 & 1 & 1 & 1 & 1 & 0 & 1 & 1\\
 $\xi^I_t$  & 0 & 1 & 1 & 1 & 0 & 1 & 0 & 1 & 1 & 1 & 1 & 0 & 0 & 1 & 0\\
 & & $\tilde \ell_t$ & &  $\tilde \ell_t+2$ & & & & & $\tilde r_t-2$ & & $\tilde r_t$ 
& & & $\bar r_t$
\end{tabular}
\end{center}

\noindent
Deaths or births in $[\tilde\ell_t+k,\tilde r_t-k]$ cause the same changes in both processes, while deaths in $[\tilde\ell_t,\tilde\ell_t+k-1]$ or $[\tilde r_t-k+1,\tilde r_t]$ cause the boundaries to move inward. It is possible for $[\tilde\ell_t,\tilde\ell_t+k-1]$ and $[\tilde r_t-k+1,\tilde r_t]$ to overlap. When $k=3$ if $[\tilde t_t,\tilde r_t]$ consists of 3 to 5 consecutive 1s, a death can cause $\zeta^I_t$ to become $\emptyset$, but this does not contradict the coupling result.
\end{proof} 

\begin{proof}[Proof of Lemma \ref{tildesp}.] 
Recall that we suppose $\lambda >\lambda_c$. By symmetry, it suffices to prove the result for $\tilde r_t$. Since $\tilde r_t \le \bar r_t$ we
have $\limsup_{t\to\infty} \tilde r_t/t \le \alpha(\lambda)$. Suppose now that there is an $\ep>0$ and a sequence $t_n \to\infty$ with $t_{n+1}-t_n \ge t_n^{1/2}$ so that
$\tilde r(t_n) \le (\alpha(\lambda)-2\ep)t_n$ and $\bar r(t_n) \ge (\alpha(\lambda)-\ep)t_n$. 

The first step is to observe that with high probability 
$$
\tilde r(t_n+t_n^{1/2}) \le b_n \equiv (\alpha(\lambda)-2\ep)t_n + 2k \lambda t_n^{1/2}
$$
To see this, note that births at the right boundary $\tilde r(t)$ occur at rate $\lambda/2$. In the best case scenario for growth of the occupied sites  the line to the right of $\tilde r(t_n)$ consists of alternating vacant sites and occupied intervals of length $k-1$, but even in this situation if there are fewer than $\lambda n^{1/2}$ births (an event with high probability)
$$
\tilde r(t_n+t_n^{1/2}) - \tilde r(t_n) \le k \lambda n^{1/2}.
$$
The occupied sites to the right of $b_n$ cannot give birth during $[t_n,t_n+t_n^{1/2}]$. There are fewer than $\ep t_n$ of them so in time $t_n^{1/2}$ all of them will die. Thus at time $T_n = t_n + t_n^{1/2}$ we have $\bar r(T_n) \le b_n$. Since this holds for all $n$ it contradicts $\bar r_t/t \to \alpha(\lambda)$ and proves the desired result.
\end{proof}

\begin{proof}[Proof of Theorem \ref{cct}.] It is clear that $\xi^A_t \To \delta_\emptyset$ on $\{\tau^A<\infty\}$, so we only have to show that $\xi^A_t \To \xi^1_\infty$ on 
$\{\tau^A=\infty\}$. First use Lemma \ref{P2.1} to pick $n_0$ so that
$P( \xi^{[-n_0,n_0]|}_t \neq \emptyset$ for all $t\ge 0)> 1-\delta$.

Stealing a line from the end of the proof of Proposition 3.1; choose $N$ so large so that any $N$ points in $\ZZ$ will contain a subset of size $N'$ separated by a distance of $2n+1$, where $N'$ is such that in $N'$ independent trials with success probability $P({}_n \xi_1^0 \supset [-n,n])$ we will have at least one success with probability at least $1-\delta$. 

By the version of  Lemma \ref{P2.2} that was proved for the $k$-creation model 
$$
\lim_{t\to\infty} P(|\zeta^A_t| \ge N) = P^A(\Omega_\infty)
$$
Pick $t$ so that $P(|\zeta^A_t| \ge N)  \ge (1-\delta) P^A(\Omega_\infty)$.
Combining this result with the previous paragraph we see that with probability
$\ge (1-\delta)^2 P^A(\Omega_\infty)$ we have a translate of  $[-n_0,n_0]$ contained in $\xi^A_{t+1}$. By the coupling in Lemma \ref{kcouple}, we see that when the process starting from the translate of $[-n_0,n_0]$ lives forever the $\xi^A_t \To \xi^1_\infty$. 

Combining our estimates we see that convergence to $\xi^1_\infty$ holds on a set of probability $(1-\delta)^3  P^A(\Omega_\infty)$. Since $\delta$ is arbitrary the desired result follows.
\end{proof}

\clearp

\section*{References}

\mn
Bezuidenhout, C., and Gray, L. (1994)
Critical attractive processes.
{\it Ann. Probab.} 22, 1160--1194

\mn
Bezuidenhout, C., and Grimmett, R. (1990)
The critical contact process dies out.
{\it Ann. Probab.} 18, 1462--1482

\mn
Dickman, R., and Tom\'e, T. (1991)
First order phase transitions in a one-dimensional nonequilbrium model.
{\it Phys. Rev. A.} 44, 4833--4838

\mn
Durrett, R. (1980)
 On the growth of one dimensional contact processes. 
{\it Ann. Probab.} 8), 890--907

\mn
Durrett, R. (1984)
 Oriented percolation in two dimensions. Special Invited Paper.
{\it  Ann. Prob.} 12, 999--1040

\mn
Durrett, R. (2019)
{\it Probability: Theory and Examples.}
Cambridge U. Press

\mn
Durrett, R. (2026) 
Discontinuous phase transitions in the one-dimensional pair and triplet creation processes 
{\it arXiv}:2608.29878

\mn
Durrett, R. (2027)
Interacting Particle Systems: Ideas, Techniques, Applications
\url{https://sites.math.duke.edu/~rtd/PASTA/PASTAcon.html}

\mn
Gray, L. (1986) 
Duality for general attractive spin systems with applications in one dimension.
{\it Ann. Probab.} 14, 371--396

\mn
Griffeath, D. (1979)
{\it Additive and Cancellative Interacting Particle Systems.}
Springer Lecture Notes in Math 724

\mn
Grimmett, G. (1999) 
{\it Percolation.} Second Edition,
Springer, New York

\mn
Harris, T.E. (1976)
On a class of set-valued Markv processes.
{\it Ann. Probab.} 4, 175--194

\mn
Harris, T.E. (1977)
A correlation inequality for Markov processes on partially ordered spaces.
{\it Ann. Probab.} 5, 451--454

\mn
Liggett, T.M. (1999)
{\it Stochastic Intereacting Systems: Contact, Voter and Exclusion Processes.}
Springer-Verlag, New York

\mn
Liggett, T.M., Schonmann, R.H., and Stacey, A.M. (1997) 
Domination by product measures.
{\it Ann. Probab.} 25, 71–95

\end{document}